\documentclass[a4paper,11pt]{article}

\usepackage{amsmath,amsfonts,amssymb,amsthm}
\usepackage{mathrsfs}
\usepackage{enumitem}

\author{C S Preenu\\University College Thiruvananthapuram\\Research Centre University of Kerala\and A R Rajan\\Former Professor of Mathematics\\University of Kerala\and K S Zeenath \\Former Professor of Mathematics\\School of Distance  Education \\University of Kerala}
\title{Partial Order in the  Normal Category Arising from Normal Bands}
\date{}

\newcommand{\gl}{\ensuremath{\mathscr{L}}} 
\newcommand{\gr}{\ensuremath{\mathscr{R}}}

\newcommand{\lb}{\ensuremath{\mathcal{L}(B)}}
\newcommand{\los}{\ensuremath{\mathcal{L}(S)}}

\theoremstyle{definition}
\newtheorem{thm}{Theorem}[section]
\newtheorem{prop}[thm]{Proposition}

\newtheorem{defn}{Definition}[section]

\begin{document}
	\maketitle
\begin{abstract}
	The notion of normal category was introduced by KSS Nambooripad  in connection with the study of the structure of regular semigroups using cross connections\cite{nambooripad1994theory}. It is an abstraction of the category of principal left ideals of a regular semigroup. A normal band is a semigroup $B$ satisfying $a^2=a$ and $abca=acba$ for all $a,b,c \in B$. Since the normal bands are regular semigroups, the category $\lb$ of principal left ideals of a normal band $B$ is a normal category. One of the special properties of this category is that the morphism sets admit a partial order compatible with the composition of morphisms. In this article we derive several properties of this partial order and obtain a new characterization of this partial order. 
\end{abstract}

\section{Preliminaries}
In this section we discuss some basic concepts of normal categories,  normal bands and the normal category derived from normal bands. We follow \cite{mac2013categories,nambooripad1994theory} for other notations and definitions which are not mentioned.  
\subsection{Normal Categories}
We begin with the category theory, some basic definitions and notations. In a category $C$, the class of objects is denoted by $vC$ and $C$ itself is used to denote the morphism class. The collection of all morphisms from $a$ to $b$ in $C$ is denoted by $C(a,b)$ and $1_a$ denote the identity morphism in $C(a,a)$. A morphism $f\in C$ is said to be a monomorphism(epimorphism) if $fg=fh$($gf=hf$) implies $g=h$ for any two morphisms $g,h\in C$. A morphism $f\in C(a,b)$ is said to be right(left) invertible if there exists $g\in C(b,a)$ such that $fg=1_a$($gf=1_b$). An isomorphism means a morphism which is both left and right invertible.

Let $S$ be a regular semigroup. A right translation of $S$ is a function  $\phi$ from $S$ to $S$ such that $(xy)\phi=x((y)\phi)$ for all $x,y\in S$. Similarly  a function $\psi:S\to S$ is called a left translation if $\psi(xy)=(\psi(x))y$ for all $x,y\in S$.  For a fixed $u\in S$, the map $\rho_u:x\mapsto xu$ for $ x\in S$ is a right translation. This map can be viewed as a function from $S$ to the principal left ideal $Su$ of $S$. Since $S$ is regular, we can find an idempotent $f\in S$ such that  $Su=Sf$. In particular if we choose $u\in eSf$ the restriction of $\rho_u$ to $Se$, denoted by $\rho(e,u,f)$, can be viewed as a function from $Se$ to $Sf$. If $\rho(e,u,f):Se\to Sf$ and $\rho(g,v,h):Sg\to Sh$ are such that $Sf=Sg$, the composition of these maps will be $\rho(e,uv,h)$. Now we arrive  at a category structure, $\los$, with objects are the principal left ideals of $S$ and morphisms are of the form $\rho(e,u,f):Se\to Sf$ where $u\in eSf$.  The usual set inclusion of the principal left ideals will induce a partial order in the objects class of $\los$. Also  $Se\subseteq Sf$ implies $ef=e$ and hence the morphism $\rho(e,e,f)$ is an inclusion map from $Se$ to $Sf$ and is denoted by $j(Se,Sf)$. The category $\los$  become a normal category whose definition is given below.  

A preorder is a category $P$ such that  each of the hom sets contain at most one morphism.  We write $a\le b$  whenever $P(a,b)\ne \emptyset$. 
\begin{defn}[\cite{rajan2015normal}]
	A category with normal factorization is a small category $C$ with the following properties. 
	\begin{enumerate}
		\item The vertex set $vC$ of $C$ is a partially ordered set such that whenever $a\le b$ in $vC$, there is a monomorphism $j(a,b):a\to b$ in $C$. This morphism is called the inclusion from $a$ to $b$.
		\item $j:(vC,\le)\to C$ is a functor from the preorder  $(vC,\le)$ to $C$ which maps $a\le b $ to $j(a,b)$ for $a,b\in vC$.
		\item For $a,b\le c$ in $vC$, if $$j(a,c)=fj(b,c)$$ for some $f:a\to b$ then $a\le b$ and $f=j(a,b)$.
		\item Every morphism $j(a,b):a\to b$ has a right inverse $q:b\to a$ such that $$j(a,b)q=1_a $$ Such a morphism $q$ is called a retraction in $C$.
		\item Every morphism $f$ in $C$ has a factorization $$f=quj$$ where $q$ is a retraction, $u$ is an isomorphism and $j$ is an inclusion. Such a factorization is called normal factorization in $C$.
	\end{enumerate}
\end{defn}

\begin{prop}[\cite{nambooripad1994theory}]
	Let $C$ is a category with normal factorization. If $f=quj=q'u'j'$ are two normal factorizations of $f\in C$ then $j=j'$ and $qu=q'u'$.

	In this case $f^0=qu$ is called the epimorphic component of $f$. The codomain of $f^\circ$ is called the image of $f$ and is denoted by $\text{im}\,f$  
\end{prop}

Corresponding to a right translation $\rho_u$ from $S$ to $Su$, we get  morphisms from each of the principal left ideals of $S$ to $Su$ by restricting the $\rho_u$. This phenomena  is similar to the concept of cones in the category theory\cite{mac2013categories}. Now we define normal cones  and  normal categories.

\begin{defn}\cite{nambooripad1994theory}
	Let $C$ be a category with normal factorization. A cone $\gamma$ with vertex $a\in vC$ is a function from $vC$ to $C$ satisfying the following
	\begin{enumerate}
		\item $\gamma(c)\in C(c,a)$ for all $c\in vC$
		\item If $c_1\subseteq c_2$ then $\gamma(c_1)=j(c_1,c_2)\gamma(c_2)$
		\item There is an object $b\in vC$ such that $\gamma(b)$ is an isomorphism. 
	\end{enumerate} 
\end{defn}

\begin{defn}[\cite{rajan2018inductive}]
	A normal category is a category with normal factorization such that for each $a\in vC$ there is a normal cone $\gamma$ with $\gamma(a)=1_a$.
\end{defn}

The cones induced by a right translation $\rho_u$ is called a principal cone and is denoted by $\rho^u$. The category $\los$ may contain normal cones other than  principal cones.
\subsection{Normal Bands}
Several studies have been carried out for recent years about the normal categories arising from the semigroups satisfying some special conditions\cite{rajan2015normal,muhammed2016normal,pcscatnormal2019}.  We  focus on the normal category of principal left ideals of a normal band.  A semigroup $B$ is said to be a normal band if $a^2=a$  and $abca=acba$ for all $a,b,c\in B$.   The following theorem, by Yamada and Kimura\cite{kimura1958structure}, states that a semigroup is a normal band if and only if it is a strong semilattice of rectangular bands.

\begin{thm}[\cite{kimura1958structure,yamada1958note}]\label{structure-nband}
	The following conditions are equivalent for a semigroup $B$
	\begin{enumerate}
		\item $B$ is a normal band
		\item There exists a semilattice $\Gamma$, a disjoint family of rectangular bands $\{E_\alpha:\alpha\in\Gamma\}$ such that $B=\bigcup\{E_\alpha:\alpha\in\Gamma\}$ and  a family of homomorphisms  $\{\phi_{\alpha,\beta}:E_\alpha\to E_\beta,\text{ where }\alpha,\beta\in \Gamma\text{ and }\beta\le \alpha \}$ satisfying the following
		\begin{enumerate}
			\item $\phi_{\alpha,\alpha}$ is the identity function,
			\item $\phi_{\alpha, \beta}\phi_{\beta, \gamma}=\phi_{\alpha, \gamma}$ for $\gamma\le\beta\le\alpha$ and 
			\item For $a,b\in B$ with $a\in E_\alpha$ and $b\in E_\beta$
			$$ab=(a\phi_{\alpha,\alpha\beta})(b\phi_{\beta,\alpha\beta}) $$
		\end{enumerate}
	\item $B$ is a band and  
	$$abcd=acbd \text{ for all }a,b,c,d\in B.$$
	\end{enumerate}

\end{thm}

 \subsection{The category $\lb$}
Let $B$ is a normal band. In addition to that the category $\lb$ is normal, it possess some special qualities which are not common in general. Several properties of the category $\lb$ are described in \cite{pcscatnormal2019}. Here the normal factorization is unique and there is only one isomorphism between any two isomorphic objects. Also the principal cones respect the isomorphic objects in the sense that 
$$\gamma(Ba)=\rho(a,ab,b)\gamma(Bb) $$
for any principal cone $\gamma$ and isomorphic objects $Ba$ and $Bb$. This leads to the definition of strong cones by AR Rajan\cite{rajan2015bands}.
\begin{defn}\cite{rajan2015bands}
	A normal cone $\gamma$ in a normal category $C$ is said to be a strong cone if the following holds
	
	If $c$ and $d$ are two isomorphic objects, then there is a unique  isomorphism $\alpha$ from $c$ to $d$ such that $\gamma(c)=\alpha\gamma(d)$
\end{defn}
Thus every principal cone is a strong cone. That is for every objects $Ba$ of $\lb$ there is a strong cone with vertex $Ba$. The category of principal left ideals of any normal band is   characterized as follows.   
\begin{thm}\label{defthm}\cite{pcscatnormal2019}
	A normal category $C$  satisfies  following conditions if and only if it is isomorphic to $\lb$ for some normal band $B$ 
	\begin{enumerate}[label=\bf SC\arabic*]
		\item\label{sc0} The isomorphism between any two objects is unique.
		\item\label{sc1} The right inverse of an inclusion is unique.
		\item\label{sc2} For every object $c$ in $C$, there is a strong cone with vertex $c$.
	\end{enumerate}
\end{thm}
\section{Partial order in $\lb$}
Let $B$ is a normal band. Let us denote by $[Ba,Bb]$, the collection of all morphisms from $Ba$ to $Bb$ in $\lb$. Now we show that there is a partial order on the morphism sets $[Ba, Bb]$
and that $[Ba, Bb]$ is an order ideal with respect to this partial order. In
particular there is a maximum element in $[Ba, Bb]$. We note that any normal
band $B$ has a natural partial order given by $a\le  b$ if and only if $ab = ba = a$.
\begin{thm}\cite{pcscatnormal2019}
	Let $B$ be a normal band and $\rho(a, u, b)$, $\rho(a, v, b) \in  [Ba, Bb]$ be morphisms in $\lb$. Define 	$$\rho(a, u, b) \le  \rho(a, v, b)\text{ if } u \le v$$ Then $\le$ is a partial order on $[Ba, Bb]$.
\end{thm}

\begin{prop}
	Every morphism set $[Ba,Bb]$ contains a maximum with respect to the above partial order.
\end{prop}
\begin{proof}
	Taking $\delta=\alpha\beta$ where $a\in E_\alpha$ and $b\in E_\beta$ we see that there is a unique $u\in E_{\alpha\beta}$ such that $\rho(a,u,b)\in [Ba,Bb]$.
	
	Since $\omega^l(b)\cap \omega^r(a)=\omega(h)$ for some $h\in E_{\alpha\beta}$ and $\rho(a,h,b)\in[Ba,Bb]$ we conclude that $u=h$. It follows that $\rho(a,u,b)$ is the maximum in $[Ba,Bb]$ 
\end{proof}

The following compatibility properties hold for this partial order.
\begin{prop}\label{compatibility}
	Let $\rho(a, u, b)$ and $\rho(a, v, b)$ be elements in the morphism
	set $[Ba, Bb]$ in $\lb$ such that $\rho(a, u, b) \le \rho(a, v, b)$. Then the following compatibility relations hold.
	\begin{enumerate}
		\item If $\rho(b, w, c) \in [Bb, Bc]$ then $\rho(a, u, b)\rho(b, w, c) \le  \rho(a, v, b)\rho(b, w, c)$.
		\item  If $\rho(d, x, a) \in [Bd, Ba]$ then $\rho(d, x, a)\rho(a, u, b)\le  \rho(d, x, a)\rho(a, v, b)$.
	\end{enumerate}
\end{prop}
\begin{proof}
	Follows from the compatibility of natural partial order on normal bands.
\end{proof}

Now we describe some properties of the maximum elements in the morphism sets.

\begin{prop}\label{iso-max}
	Let $m = m(a, b)$ be the maximum element in the morphism set $[Ba, Bb]$ in $\lb$ and $\alpha = \alpha(b, c)$ be an isomorphism from $Bb$ to $Bc$
	in $\lb$. Then $m\alpha$ is the maximum element in the morphism set $[Ba, Bc]$.
\end{prop}
\begin{proof}
	Let  $f \in  [Ba, Bc]$. Then $f\alpha^{-1}\in  [Ba, Bb]$. Since $m$ is maximum element in $[Ba, Bb]$ we have $f\alpha^{-1}\le m$. Now by compatibility given by Proposition \ref{compatibility}
	$$f = f\alpha^{-1} \alpha \le m\alpha$$.
	So $m\alpha$ is the maximum in $[Ba, Bc]$. 
\end{proof}
\begin{prop}
	Let $m = m(a, b)$ be the maximum element in the morphism set $[Ba, Bb]$ in $\lb$ and $j = j(a_0,a)$ be an inclusion in $\lb$. Then
	$jm$ is the maximum element in the morphism set $[Ba_0,Bb]$.
\end{prop}
\begin{proof}
Let $f \in  [Ba_0,Bb]$. Then $f=\rho(a_0,u,b)$ for some $u \le k$ such that
$\omega(k)=a_0Bb$. Now $$jm =\rho(a_0,a_0,a)\rho(a,h,b) = \rho(a_0,a_0h,b) = \rho(a_0,k,b)$$
where $\omega(h) = aBb$. Therefore $f \le jm$ so that $jm$ is maximum element in $[Ba_0,Bb]$.
\end{proof}

\begin{prop}\label{retraction}
	Let $m = m(a, b)$ be the maximum element in the morphism set $[Ba, Bb]$ in $\lb$ and $q:b\to b_0$ be a retraction in $\lb$. Then
	$mq$ is the maximum element in the morphism set $[Ba, Bb_0]$.
\end{prop}

\begin{proof}
	Let $f\in[Ba,Bb_0]$. Then $fj(b_0,b)$ is in $[Ba,Bb]$. Since $m$ is the  maximum element in $[Ba, Bb]$ we have $fj(b_0,b)\le  m$. Now by compatibility given by Proposition \ref{compatibility}
	$$f = fj(b_0,b)q \le  mq.$$
	So $mq$ is the maximum in $[Ba,Bb_0]$.
\end{proof}

As a consequence of this propositions we can see that $mf^\circ$ is always a maximum if $m$ is a maximum.
\begin{prop}
	Let $m = m(a,b)$ be the maximum element in the morphism set $[Ba,Bb]$ in $\lb$ and $f:b \to c$ be any morphism in $\lb$. Then
	$mf^\circ$ is the maximum element in the morphism set $[Ba,Bc_0]$ where $c_0 \le c$	is the codomain of $f^\circ$.
\end{prop}
\begin{proof}
	Suppose $f=\varrho u j$, where $\varrho$ is a retraction, $u$ is an isomorphism and $j$ is an inclusion. Then by the above propositions \ref{iso-max} and \ref{retraction} we can say that 	$mf^\circ=m\varrho u$ is the maximum in $[Ba,Bc_0]$
\end{proof}

The following proposition  exactly locates the maximum element in each hom set in $\lb$
\begin{prop}\label{exact:max}
		The maximum element in $[Ba,Bb]$  is $\rho(a,ab,b)$
\end{prop}
\begin{proof}
Suppose $\rho(a,u,b)\in [Ba,Bb]$. Then $u\in aBb$ and hence $u=aub$. Now
$$(ab)(u)=(ab)(aub)=aabub=aubb=aub=u $$ Similarly we get  $$(u)(ab)=aubab=auab=aub=u$$ That is  $u\le ab$ for all $u\in aBb$ and this proves that $$ \rho(a,u,b)\le \rho(a,ab,b)$$ for all $\rho(a,u,b)\in [Ba,Bb]$.
\end{proof} 
 Note that $\rho(a,ab,b)$ is the unique isomorphism in $[Ba,Bb]$ whenever $Ba$ and $Bb$ are isomorphic. Also the following theorem shows that these morphisms are the building blocks of principal cones in $\lb$. 
\begin{thm}
The component of the principal cone $\rho^a$  at $Bb\in \lb$ is the maximum element in $[Bb,Ba]$.  
\end{thm}
\begin{proof}
	We have the principal cone $\rho^a$ with vertex $Ba$ is defined as
	$$\rho^a(Bb)=\rho(b,ba,a)\text{ for all }Bb\in v\lb $$
	
	But $\rho(b,ba,a)$ is the maximum in $[Bb,Ba]$ by proposition \ref{exact:max}. Hence the theorem.
\end{proof}
Now we give another characterization of the partial order in the hom sets of $\lb$ using the image of the morphisms.
\begin{thm}
	For any two morphisms $\rho(a,u,b),\rho(a,v,h)\in[Ba,Bb]$
	$$\rho(a,u,b)\le \rho(a,v,b)\text{ if and only if }\text{im}\,\rho(a,u,b)\subseteq \text{im}, \rho(a,v,b)  $$
\end{thm}
\begin{proof}
	To prove the theorem it is enough to show the following 
	$$ \text{im}\,\rho(a,u,b)\subseteq \text{im}, \rho(a,v,b) \text{ if and only if }u\le v$$
	
	The unique normal factorization\cite{pcscatnormal2019} of $\rho(a,u,b)$ is given by
	$$\rho(a,u,b)=\rho(a,g,g)\rho(g,u,h)\rho(h,h,b).$$
	Where $u\gr g\le a$ and $u\gl h\le b$. Therefore $\text{im}\, \rho(a,u,b)=Bh=Bu$. Similarly we  have $\text{im}\, \rho(a,v,b)=Bv$. 
	
	Assume that  $u\le v$. Then $uv=vu=u$. This implies $Bu=Buv\subseteq Bv$. Thus $\text{im}\,\rho(a,u,b)\subseteq \text{im}\, \rho(a,v,b)$. 
	
	Conversely suppose that $ \text{im}\,\rho(a,u,b)\subseteq \text{im}, \rho(a,v,b)$. That is $Bu\subseteq Bv$. Then $uv=v$. 
	
	Note that $u,v\in aBb=aB\cap Bb$. Therefore there exists $p,q,r,s\in B$ such that $u=ar=pb$ and $v=as=qb$. Also $u=u^2=arpb$ and $v=v^2=asqb$. Then $uv=u$ implies $arqb=ar$ and $pbqb=pb$. Now we get 
	\begin{align*}
	vu&=(asqb)(arpb)\\
	&=(aras)(pbqb)&\text{By repeted application of normality}\\
	&=(ar)(pb)\\
	&=uu\\&=u
	\end{align*} 
	Therefore $ \text{im}\,\rho(a,u,b)\subseteq \text{im}, \rho(a,v,b)$ implies $uv=vu=u$ and hence $u\le v$. Hence the theorem. 
	
\end{proof}

\end{document}